\documentclass[10pt,reqno]{amsart}
\usepackage{amssymb,amsmath,bm,geometry,graphics,url,color}
\usepackage{amsfonts}
\usepackage{graphicx}
\usepackage{epsfig,multirow}
\usepackage{amscd}
\usepackage{enumerate}
\usepackage{mathrsfs}
\usepackage{xypic}
\usepackage{stmaryrd}
\usepackage{fancybox}
\usepackage{exscale,array}
\usepackage{enumitem}%
\usepackage{color,cases}
\usepackage{hyperref}
\usepackage{algorithm}
\usepackage{algorithmic}
\usepackage{verbatim}
\usepackage{amsmath}
\usepackage{supertabular}
\usepackage{array}
\usepackage{multirow}
\usepackage{booktabs}
\usepackage{threeparttable}
\def\pdt2{\partial_t^2}
\def\pdx2{\partial_x^2}

\def\eps{\varepsilon}

\newcommand{\bv}{{\bf v}}

\newcommand{\bw}{{\bf w}}

\newcommand{\HH}{{\bf H}}
\newcommand{\EE}{{\bf E}}
\newcommand{\FF}{{\bf F}}
\newcommand{\CU}{{\bf curl}}

\newcommand{\ba}{\mathbf{a}}
\newcommand{\bx}{\mathbf{x}}

\newcommand{\br}{{\mathbf r}}

\newcommand{\bF}{\mathbf{F}}
\newcommand{\bJ}{\mathbf{J}}

\newtheorem{theo}{Theorem}[section]
\newtheorem{lem}[theo]{Lemma}
\newtheorem{cor}[theo]{Corollary}
\newtheorem{rem}[theo]{Remark}

\newtheorem{prop}[theo]{Proposition}

\numberwithin{equation}{section}

\title[Closed form solution of Maxwell's equations]{An attractive analytic solution of the Maxwell's equation}
\author[X.R. Zou]{Xiaorong Zou}\address{\hspace*{-12pt}X.R.~Zou: Global Market Risk Analytic, Bank of America, NYC, NY, USA}
\email{xiaorong.zou@bofa.com}

\begin{document}
\maketitle
\dedicatory{}
\begin{abstract}
	In this paper,  we aim to develop a close-form solution of the standard Maxwell's equations that  describe the propagation of electromagnetic waves in an isotropic homogeneous medium such as a vacuum. The proposed solution has a simple analytic structure that explains how initial conditions characterize the underlying electromagnetic wave.  As such, it can be used as a tool to generate suitable electromagnetic waves in practice. The clean structure of the solution can also help us to derive  closed-form solutions in a general medium with non-zero current term.  In this paper, we enhance the solution to cover a special case where the current comes from the contribution of an independent generator.  We shall leverage the tools developed in this paper to derive similar closed-form solutions of Maxwell's equations in general mediums with non-zero currents in a separate paper. 
\end{abstract}
 \section{Introduction}
The Maxwell's equations are the foundation of classical electromagnetic waves and play a critical role in a wide variety of applications in science and engineering.   In this paper,  we aim to develop a close-form solution of the standard Maxwell's equations that characterize the propagation of electromagnetic waves in an isotropic homogeneous medium such as a vacuum.
 \begin{subequations}\label{maxsys}
\begin{numcases}%
\,\frac{\partial \HH}{\partial t} =-\frac{1}{\mu}\CU\  \EE,\qquad \frac{\partial \EE}{\partial t} =\frac{1}{\eps}\CU\  \HH , \label{maxsys a}\\
\HH(\bx,0) = H_0(\bx) ,\qquad  \EE(\bx,0) =E_0(\bx), \label{IV}
\end{numcases}
\end{subequations}
where $\bx = (x_1,x_2,x_3)\in R^3$, $\textbf{H}= (H_1 ,H_2 ,H_3 )^{\intercal}: \mathbb{R}^3\times
\mathbb{R}_+ \rightarrow \mathbb{R}^3$ {stands for}
magnetic field intensity,
$\textbf{E}= (E_1 ,E_2 ,E_3 )^{\intercal}: \mathbb{R}^3\times
\mathbb{R}_+ \rightarrow \mathbb{R}^3$ represents { electric field intensity}, $H_0(\bx), E_0(\bx)$ represents the initial condition of the system,  $\mu$ and $\eps$ are constant in this paper and characterize the magnetic permeability and electric permittivity, respectively.  It is well known that Maxwell's equations \eqref{maxsys} have unique smooth solutions for all time if the initial data  \eqref{IV} are suitably smooth \cite{Leis86}.

 There are rich researches on numerical solutions of Maxwell's equations over the last couple of decades due to the importance and a wide range of applications. The standard finite-difference time domain (FDTD) method  was initially proposed in { Yee  in \cite{3}}, and further enhanced in \cite{Monk94,n38,n49,n41,n43,n44}. In general, FDTD-based methods require small temporal step-size to ensure their stability.  The alternating direction implicit (ADI) technique was proposed to improve the efficiency in {\cite{3,4,5,6,n25,n22}}.  Time dimension $t$ and space dimension $x,y,z$ are often treated separately in the literature.  In space dimension, in addition to the standard finite difference method,  discontinuous Galerkin (dG) \cite{n33,Moya12,Descombes13,Descombes16} and Fourier pseudo-spectral \cite{n29} have been developed.  The spatially {discretised} Maxwell's equations are often transformed to integration in time dimension that can be solved by various schemes such as  the explicit method \cite{Grote10},  Verlet method \cite{Fahs09},
Runge-Kutta (RK) methods { \cite{Hochbruck151,Hochbruck16}}, splitting methods \cite{Verwer11, Hochbruck15},  low-storage RK schemes \cite{Diehl10}, multiscale methods \cite{Henning16, Hochbruck19}, and exponential RK {methods \cite{t13,yang}}. The solutions of Maxwell's equations \eqref{maxsys} exhibit quite a few  physical invariants: energy
conservation laws, symplectic  conservation laws, helicity
conservation laws, momentum conservation laws and divergence-free
fields. These invariants are very important in the long time
propagation of the electromagnetic waves \cite{n44}.  In \cite{WJ}, an efficient scheme is proposed to solve the Maxwell’s equations. It achieves decent accuracy and keeps the desired conservation laws in discretised states. 

To author's best knowledge,  there are little researches in literature that consider exact analytic solutions of Maxwell’s equations. In this paper,  we aim to develop an analytic solution of Equation \eqref{maxsys}.  First, we apply Rieman--Silberstein representation of the electromagnetic fields and transform Maxwell Equation \eqref{maxsys} to an abstract Cauchy problem (ACP) \cite{engel_nagel}.  We then study the properties of eigenvalues and eigenvectors of the underlying linear operator of ACP, and use them to derive desired close-formed solution.

The proposed solution has a simple analytic structure that explains how initial conditions characterize the underlying electromagnetic wave. For example, one can directly verify the conservation of energy (Proposition \ref{prop:property_maxwell_1}) and divergence (Corollary \ref{cor:conservation-divergence}).  We also use it to characterize a perpendicular electromagnetic field (see Corollary \ref{cor:2}), which maximizes the power flow (Poynting vector) of the electromagnetic field with given total energy. We expect that the explicit formulation of the solution in initial conditions can be used to generate suitable electromagnetic waves in practice.  

The clean structure of the solution's in an isotropic homogeneous medium  can help us to derive  closed-form solutions under more general setting by adding current density $\bJ$ such that Equation (\ref{maxsys a}) is replaced by 
\begin{equation}\label{assumption:jg}
	\frac{\partial \HH}{\partial t} =-\frac{1}{\mu}\CU\  \EE, \qquad \frac{\partial \EE}{\partial t} =\frac{1}{\eps}\CU\  \HH - \bJ .
\end{equation}
In this paper, we consider a special case $\bJ = J^g(\bx)$, where $J^g$ is independent on time $t$ and stands for the contribution of a stationary generator and can be considered as imposed, independently of the electromagnetic field \cite{bossavit} (page 10-11). It turns out that the closed form solution of Equation (\ref{assumption:jg}) can be derived by the solution of Equation \eqref{maxsys a} with some adjustment (Theorem \ref{theorem-main-3}). The solutions of Maxwell's equation under a general setting on $\bJ$ will be studied in a separate paper \cite{zou_maxwell}.

The rest of the paper is organized as follows. In section \ref{sec:bj=0}, we first convert Maxwell Equation \eqref{maxsys} into an ACP and study the properties of eigenvalues and eigenvectors of underlying linear operator, and then use the results to derive desired close-formed solution.  Section \ref{sec:jg} is used to derive the analytic solution of Maxwell Equation \ref{assumption:jg}. In
Section \ref{sec:example},  we demonstrate how to solve Equation \eqref{maxsys} with two specific initial conditions by using the closed-form formula developed in this paper.  Some computational details are shown in Appendix \ref{app:property_maxwell_1}.

\section{The development of analytic solution of Maxwell's equation }\label{sec:bj=0}
In this section, we develop the desired closed-form solutions of Maxwell Equation \eqref{maxsys}. We rewrite two curl-equations in Equation \eqref{maxsys} as an abstract Cauchy problem (ACP).
\begin{equation}\label{eq:acp}
\begin{pmatrix}
\frac{\partial{\sqrt{\eps}  \EE}}{\partial t}    \\
\frac{\partial {\sqrt{\mu}\HH}}{\partial t}
\end{pmatrix}
=\mathcal{C} \begin{pmatrix}
	\sqrt{\eps} \EE \\
\sqrt{\mu} \HH 
\end{pmatrix},   \quad
\begin{pmatrix}
	 \sqrt{\eps}  \EE(\bx,0) \\
	  \sqrt{\mu}\HH(\bx,0)  
\end{pmatrix}
= \begin{pmatrix}
	\sqrt{\mu} H_0(\bx)  \\
	\sqrt{\eps} E_0(\bx)
\end{pmatrix},
\end{equation}
where
\[
\mathcal{C}= \left(
\begin{array}{cc}
	\mathbf{0} &  \frac{\CU}{\sqrt{\mu\eps}}\\
	-\frac{\CU}{\sqrt{\mu\eps}} & \mathbf{0} \\
\end{array}
\right)=:
\left(
\begin{array}{cc}
	\mathbf{0} &  A\\
	-A & \mathbf{0} \\
\end{array}
\right).
\]
Follow Riemann-Silberstein representation of the electromagnetic fields \cite{jestada}, let $\FF = \sqrt{\epsilon}\EE + \sqrt{\mu}\HH \cdot i$, and we have 
\[
\frac{\partial \bF}{\partial t} = \frac{1}{   \sqrt{\mu\epsilon}} \CU (\sqrt{\mu}\HH) -  \frac{1}{\sqrt{\mu\epsilon}} \CU (\sqrt{\epsilon}\EE)i = -i \frac1 {\sqrt{\mu\epsilon}}  \CU (\sqrt{\epsilon} \EE +  \sqrt{\mu}\HH i) =  -i A \FF.
\]
ACP (\ref{eq:acp}) is reduced to 
\begin{equation}\label{maxwell_complex}
\frac{\partial \FF}{\partial t} = -i  A \FF,  \qquad  \FF (\bx,0) = \sqrt{\epsilon}\EE_0  +  \sqrt{\mu}\HH_0 \cdot i .
\end{equation}
In general, one can solve the ACP \cite{engel_nagel}
\begin{equation}\label{eq-acp}
	F(\bx, t) = e^{ -i  A t}F(\bx,0) = \sum_{n=0} \frac{(-itA)^{(n)}}{n!}F(\bx,0) .
\end{equation}
Throughout the paper, $\bw =(w_1,w_2, w_3)^T\in R^3$ denote a fixed vector, also called wave vector in literature,   $V^{\bw}=e^{i\bx \cdot \bw} C^3$ be the $3$-dim complex vector space generated by $e^{i\bx \cdot \bw}$. Define 
\begin{eqnarray*}
s_{\bw} &:=& w_1+w_2+w_3, \\
{\br}_{\bw} &=& (w_2-w_3, w_3-w_1, w_1-w_2)^T  \\
\gamma_{\bw} &:=& |{\br}_{\bw}|  = \sqrt{3|\bw|^2 - s_{\bw}^2}, \quad  \nu_{\bw} = \sqrt 2\frac{\gamma_{\bw} }{|\bw|}.
\end{eqnarray*}
The following observation is the key to establish a closed-form solution of ACP (\ref{maxwell_complex}).
%For a more explicit analytic solution, we study the properties of the spectrum of the generator $iA$. 
\begin{lem} The restriction of $\CU$ on $V^{\bw}$ has three eigenvalues  $\{\mu_{d,\bw}\}_{d=1}^3 $ with
\begin{equation}\label{eq:eigenvalue}
\mu_{1,\bw} = 0, \quad \mu_{2,\bw} = -|\bw|, \quad \mu_{3,\bw} = -\mu_{2,\bw},
\end{equation}
where 
\[
|\bw| = \sqrt{w_1^2 + w_2^2 +w_3^2}.
\]
\end{lem}
\begin{proof} 
Define the anti-symmetric matrix $\Phi_{\bw}$ by
%\begin{equation}\label{eq:phi}
\[
\Phi_{\bw}=\left(
\begin{array}{ccc}
	0 & -w_3 &w_2\\
	w_3 & 0 & -w_1 \\
	-w_2 & w_1 & 0
\end{array}
\right).
\]
%\end{equation}
Let $\eta =(\xi_1,\xi_2,\xi_3)^Te^{i\bx \cdot \bw}=\xi e^{i\bx \cdot \bw}$. By definition of $\CU$, we have
\begin{equation}\label{eq:curl}
	\CU (\eta) = i  \Phi_{\bw}\xi e^{i\bx \cdot \bw}.
\end{equation}
If $\eta$ is a eigenvector of the operator $\CU$ associated to certain eigenvalue $i\lambda$, i.e.
\[
i\lambda \xi e^{i\bx \cdot \bw} = \CU (\eta) =  i  \Phi_{\bw}\xi e^{i\bx \cdot \bw},
\]
or equivalently,
\begin{equation}\label{eq:Phi_w}
 \CU (\xi e^{i\bx \cdot \bw})= i\lambda \xi e^{i\bx \cdot \bw}  \quad \mbox{if and only if} \quad	\Phi_{\bw}\xi = \lambda \xi,
\end{equation}
which implies $\lambda$ is an eigenvalue of the skew-symmetric $\Phi_{\bw}$ whose eigenvalues are $0,i|\bw|, -i|\bw|$, implies $\mu_d$ is given by Equation (\ref{eq:eigenvalue}) for $d=1,2,3$.
%If $\{v_d\}^3_{d=1}$ are a set of eigenvalues associated to $\{i\lambda_{d,\bw}\}^3_{d=1}$. It is clear that $\{v_d e^{i\bx \cdot \bw}\}_{d=1}^3$ forms a base of $V^w$
\end{proof}

\begin{lem}\label{lem-eigenvalues}  Let $|\bw|>0$. 
	\begin{enumerate}
	\item Assume $|{\br}_{\bw}|>0$.  Define
	%\begin{equation} \label{eq:v_1}
	\[
	\bv_{1,\bw} = \bw/|\bw|
	\]
	%\end{equation}
	and $\{{\bv}_{d,\bw}\}_{d=2}^3$ by
\begin{equation}\label{v_d_w_1}
{\bv}_{d,\bw} = \frac{1}{\nu_{\bw}}(\frac{s_{\bw}}{|\bw|^2} w-(1,1,1)^T - i\frac{\mu_{d,\bw}}{|\bw|^2} {\br}_{\bw} ).   
\end{equation}
	Let 
	%\begin{equation}\label{eq:V_w}
	\[
	V_{\bw}:=(\bv_{1,\bw}, \bv_{2,\bw},\bv_{3,\bw}), \quad V_{\bw}^{\dagger}:= \bar{V}_{\bw}^T.
	\]
		%\end{equation}
 Then   
	\begin{eqnarray}
		V_{\bw} \cdot V_{\bw}^{\dagger} &=& I_3, \label{eq:normality_general}\\	
		\CU (\bv_{d,\bw} e^{i\bx \cdot \bw}) &=&\mu_{d,\bw} \bv_{d, \bw} e^{i\bx \cdot \bw}, \quad d=1,2,3,  \label{eq:eigen_vector_general}
	\end{eqnarray}
where $\{ \mu_{d,\bw}\}_1^3$ is defined by Equation (\ref{eq:eigenvalue}).
	So $\{{\bv}_{d,\bw}e^{i\bx \cdot \bw}\}^3_{d=1}$ are three  eigenvectors of $\CU$ operator associated to $\mu_{d,\bw}$ and $\{{\bv}_{d,\bw}\}^3_{d=1}$ forms a orthogonal base of $C^3$. 
	\item Assume $|{\br}_{\bw}|=0$. In this case, $w_1=w_2=w_3:= \omega$.  Let 	\begin{equation}\label{v_d_w_2}
		v_1 = \frac{1}{\sqrt 3}(1,1,1)^T,\quad {v}_{2} := \frac{1}{\sqrt{3}}(\frac{1+i\sqrt 3}2, \frac{1-i\sqrt 3}{2}, -1)^T, \quad v_3 = \bar{v}_2.
	\end{equation}  
	%Then $\{{\bv}_{d}e^{i\bx \cdot \bw}\}^3_{d=1}$ are three eigenvectors associated the eigenvalues $\CU$ on 
	Define
\[
V = (v_1,v_2,v_3), \quad (\eta_{1,\omega},\eta_{2,\omega},\eta_{3,\omega}) = (0, -\omega \sqrt 3, \omega \sqrt 3).
\]
Then
	\begin{eqnarray}
	V\cdot V^{\dagger} &=& I_3 \label{eq:normality_special} \\
	 \CU (v_{d} e^{i\bx \cdot \bw}) &=& \eta_{d,\omega}v_{d} e^{i\bx \cdot \bw} , \quad d=1,2,3 \label{eq:eigen_vector_special}
	\end{eqnarray}
\end{enumerate}
\end{lem}
\begin{proof} 
We first assume that $|{\br}_{\bw}|>0$.  One can directly verify Equation (\ref{eq:normality_general}). By Equation (\ref{eq:curl}), we have $\CU (\bw e^{i\bx \cdot \bw}) = 0$, which implies Equation (\ref{eq:eigen_vector_general}) for $d=1$.  
To verify the other two eigenvectors,  first assume that $|{\br}_{\bw}|>0$,  which implies  $v_{2,\bw}\neq v_{3,\bw}$. By Equation (\ref{eq:curl}) again, let $\lambda_{d} = -i\mu_{d,\bw}$,  we have
\[
\CU ((1,1,1)^T e^{i\bx \cdot \bw}) = i {\br}_{\bw} e^{i\bx \cdot \bw}
\]
\[
	\frac{1}{e^{i\bx \cdot \bw}}\CU ({\br}_{\bw} e^{i\bx \cdot \bw})  =   i \left( \begin{array}{c}
		w_1 w_2 +w_1w_3 - w_2^2-w_3^2 \\
		w_3 w_2 +w_1w_2 - w_1^2-w_3^2 \\
		w_1 w_3 +w_2w_3 - w_2^2-w_1^2 \\
	\end{array}\right) =   i \left( \begin{array}{c}
	w_1 w_2 +w_1w_3 + w_1^2 + \lambda_d^2 \\
	w_3 w_2 +w_1w_2 + w_2^2 + \lambda_d^2 \\
	w_1 w_3 +w_2w_3 + w_3^2 + \lambda_d^2 \\
	\end{array}\right) 
\]
which implies
\[
\CU ({\br}_{\bw} e^{i\bx \cdot \bw}) = i (s_{\bw} \bw + \lambda_d^2(1,1,1)^T)e^{i\bx \cdot \bw}.
\]
Hence
\[
\CU (\bv_{d,\bw} e^{i\bx \cdot \bw})=\frac{i}{\nu_{\bw}|\bw|^2}\{ \lambda_d^2 r_{\bw}e^{i\bx\cdot\bw} + \lambda_{d}(s_{\bw}w + \lambda_{d}^2 \begin{pmatrix}
	1\\
	1\\
	1
\end{pmatrix}e^{i\bx\cdot\bw})\}= \mu_d\bv_{d,\bw} e^{i\bx \cdot \bw}.
\]
If ${\br}_{\bw}=0$,  ${\bv}_{d,\bw}$ defined by Equation (\ref{v_d_w_1}) is not well-defined. But one can directly verify Equation (\ref{eq:normality_special}) and (\ref{eq:eigen_vector_special}) by using Equation (\ref{eq:curl}). By Equation (\ref{eq:curl}), 
\[
\CU (v_2 e^{i\bx \cdot \bw})= \frac{i \omega}{\sqrt 3}  \left( \begin{array}{c}
	-1 - \frac{1- i\sqrt 3}{2}\\
	1 + \frac{1+ i\sqrt 3}{2} \\
	 -i\sqrt 3  \\
\end{array}\right) e^{i\bx \cdot \bw} = -\sqrt 3 \omega v_2 e^{i\bx \cdot \bw} = \eta_{2,\omega} v_2 e^{i\bx \cdot \bw}.
\]
Similarly, one can verify $\CU (v_3 e^{i\bx \cdot \bw}) =\eta_{3,\omega} v_3 e^{i\bx \cdot \bw}$. 
\end{proof}
If $w_1=w_2=w_3:= \omega$,  we define
\begin{equation}\label{eq:v_special_def}
v_{2,\bw} = 
\begin{cases} 
	v_2 & \text{if } \omega \ge  0 \\ 
	v_3 & \text{if }  \omega < 0 
\end{cases}, \quad v_{3,\bw} = \bar {v}_{2,\bw}.
\end{equation}
	Then Equation (\ref{eq:eigen_vector_special}) is consistent to Equation (\ref{eq:eigen_vector_general}), i.e. Equations (\ref{eq:normality_general}-\ref{eq:eigen_vector_general}) hold for any $\bw\neq 0$ as long as $\bv_{2,\bw},\bv_{3,\bw}$ is defined by  Equation (\ref{v_d_w_1}) if ${\br}_{\bw} \neq 0$, and by Equation (\ref{eq:v_special_def}) if  ${\br}_{\bw}=0$.  For example,  if $\omega<0$ then
	\[
	\CU v_{2,\bw} e^{i\bx \cdot \bw} = \CU v_{3}e^{i\bx \cdot \bw} =  \omega \sqrt{3}  v_{3}e^{i\bx \cdot \bw} = -|\bw|  v_{2,\bw}e^{i\bx \cdot \bw} = \mu_{2,\bw} v_{2,\bw}e^{i\bx \cdot \bw}.
	\]
\begin{rem}
	In the rest of the paper, we shall apply Equation (\ref{eq:eigenvalue}, \ref{eq:normality_general}) and (\ref{eq:eigen_vector_general}) for any $\bw$  under the assumption that $\{v_{d,\bw}\}_{d=2}^3$ is defined by (\ref{v_d_w_1}) if $\br_{\bw} \neq 0$ or by (\ref{eq:v_special_def}) otherwise.  
	\begin{enumerate}
		\item Note that we have
		%\begin{equation}\label{eq:v_prop}
		\[
		v_{1,-\bw} = -v_{1,\bw},  \quad v_{2,-\bw} = v_{3,\bw}, \quad v_{3,-\bw} = v_{2,\bw}.
		\]
		%\end{equation}
		By Equation (\ref{eq:Phi_w}) and (\ref{eq:eigen_vector_general}), we have
		%\begin{equation}\label{eq:Phi_w_case}
		\[
			\Phi_{\bw} v_{d,\bw} = -i\mu_{d,\bw} v_{d,\bw} = \begin{cases} 
	0 & \text{if } d=1, \\ 
	i|\bw| v_{2,\bw}  & \text{if } d=2,\\
	-i |\bw|  v_{3,\bw} & \text{if } d=3.
\end{cases}		
		\]
		%\end{equation}
		\item 
		Let
		\[
				R_{\bw} = \frac12 (v_{2,\bw} + \bar{v}_{2,\bw}),  \quad I_{\bw} = \frac1{2i} (v_{2,\bw} - \bar{v}_{2,\bw}).
		\]
		We have
		\[
		v_{2,\bw} = R_{\bw} + iI_{\bw}, \quad v_{3,\bw} = \bar{v}_{2,\bw} = R_{\bw} - iI_{\bw}.
		\]
		One can verify directly $(R_{\bw},I_{\bw},\bw)$ forms a orthogonal base of $R^3$ for any $\bw\neq 0$ 
		\[
				\bw\cdot R_{\bw} =\bw\cdot I_{\bw}=R_{\bw}\cdot I_{\bw}=0, \quad |I_{\bw}| = |R_{\bw}| = \frac1{\sqrt 2}.
		\]
	\end{enumerate}
\end{rem}
%\section{The solution in a base subspace generated by $e^{i\bw\cdot \bx}$}
We are ready to derive the analytic solution if the initial condition is in the subspace $V^\bw$.  
\begin{prop} For a given initial function $F_{0,\bw} (\bx)=\ba_{\bw} e^{i\bx \cdot \bw}\in V^\bw$,  the solution of Equation \eqref{maxwell_complex} can be formulated as 
\begin{equation}\label{eq:formula_3}
F_{\alpha_{\bw}, \bw}(t,\bx) = \sum_{1\le d \le 3} e^{\frac{-it\mu_{d,\bw}} {\sqrt {\mu\epsilon}}} \alpha_{d,\bw} \bv_{d,\bw} e^{i\bx \cdot \bw},
\end{equation}
where	
\begin{equation}\label{eq:alpha}
	\alpha_{\bw}:=	\left( \begin{array}{c}
		\alpha_{1,\bw}\\
		\alpha_{2,\bw}\\
		\alpha_{3,\bw}\\
	\end{array}\right)=V^{\dagger}_{\bw}\left( \begin{array}{c}
		a_{1,\bw} \\
		a_{2,\bw} \\
		a_{3,\bw} \\
	\end{array}\right).
\end{equation}	
Therefore, the associated solution  $E_{\alpha_{\bw}, \bw}$ and $H_{\alpha_{\bw}, \bw}$ of Maxwell Equation \eqref{maxsys} can be formulated by 
\begin{equation} \label{eq:E_H_simple}
	E_{\alpha_{\bw}, \bw}(\bx,t) = \frac{1}2 (F_{\alpha_{\bw}, \bw}(\bx,t) + \bar F_{\alpha_{\bw}, \bw}(\bx, t)), \quad E_{\alpha_{\bw}, \bw}(\bx,t) = \frac{i}{2} (\bar F_{\alpha_{\bw}, \bw}(\bx, t) - F_{\alpha_{\bw}, \bw}(\bx,t)).
\end{equation}
\end{prop}
\begin{proof}
Since $V^{\dagger}_{\bw}$ is unitary,  we can decompose $F_{0,\bw} (\bx)$ as 
\begin{equation}\label{eq:a_represent_2}
	F_{0,\bw} (\bx)= \sum_{1\le d\le 3}\alpha_{d,\bw} \bv_{d,\bw}e^{i\bx \cdot \bw},
\end{equation}
where $\alpha_{d,\bw}$ is defined by (\ref{eq:alpha}).  By Equation (\ref{eq-acp}), we have
\begin{eqnarray*} 
	e^{-it A } (\ba_{\bw} e^{i\bx \cdot \bw}) &=& \sum_{1\le d \le 3}\sum_{n\ge 0} \frac{(-itA )^n}{n!} \alpha_{d,\bw} \bv_{d,\bw}e^{i\bx \cdot \bw}\nonumber\\
	&=& \sum_{1\le d \le 3}\sum_{n\ge 0} \frac{ (\frac{-it}{\sqrt {\mu\epsilon}})^n}{n!} (\mu_{d,\bw})^n   \alpha_{d,\bw} \bv_de^{i\bx \cdot \bw} \nonumber\\
	&=& \sum_{1\le d \le 3} e^{\frac{-it\mu_{d,\bw}} {\sqrt {\mu\epsilon}}} \alpha_{d,\bw} \bv_{d,\bw} e^{i\bx \cdot \bw}.
\end{eqnarray*}
It is easy to directly check that $F_{\bw,\alpha_{\bw}}(t,\bx)$ by Equation (\ref{eq:formula_3}) does solve (\ref{maxwell_complex}).
\end{proof}
The analytic representation (\ref{eq:E_H_simple}) enables us to establish certain properties of the solution. 
\begin{prop}\label{prop:property_maxwell_1} Let $ E_{\alpha_{\bw}, \bw}(\bx,t)$ and $H_{\alpha_{\bw}, \bw}(\bx,t)$ be defined by (\ref{eq:E_H_simple}), then 
	\begin{enumerate}
		\item $|E_{\alpha_{\bw}, \bw}(\bx,t)|^2 + |H_{\alpha_{\bw}, \bw}(\bx,t)|^2$ is constant.
		\item $ |E_{\alpha_{\bw}, \bw}(\bx,t)|$, $|H_{\alpha_{\bw}, \bw}(\bx,t)|$ and $E_{\alpha_{\bw}, \bw}(\bx,t) \cdot H_{\alpha_{\bw}, \bw}(\bx,t)$ are stationary, i.e. they are not dependent on $t$.  
	\end{enumerate}
\end{prop}
The details of proof is refereed to Appendix \ref{app:property_maxwell_1}. Note that $E^2+H^2$ represents the total energy of the  $(E,H)$, which is expected to be constant according to the conservation principle of electromagnetic field.
Since $E_{\alpha_{\bw}, \bw}(\bx,t) \cdot H_{\alpha_{\bw}, \bw}(\bx,t)$  is stationary, we have
\begin{cor} \label{cor:2}
$E_{\alpha_{\bw}, \bw}$ is orthogonal to $H_{\alpha_{\bw}, \bw}$ if they are initially orthogonal to each other, which is equivalent to the condition
\[
\alpha_{1,\bw}^2 + 2  \alpha_{2,\bw}\alpha_{3,\bw} =0. 
\]
where $\alpha_{d,\bw}$ is defined in Equation (\ref{eq:a_represent_2}) to formulate initial value $F_{0,\bw}(\bx)$. In particular, if $E_{\alpha_{\bw}, \bw}$ and $H_{\alpha_{\bw}, \bw}$ are transverse,  i.e. they are perpendicular to wave propagation direction $\bw$,  then they are perpendicular to each other if and only if $\alpha_{2,\bw}\alpha_{3,\bw}=0$.
\end{cor}
Note that the vector $\mathbf{S} = \EE \times \HH$, Poynting Vector,  represents the directional energy flux density (power per unit area) of the electromagnetic field. 

%\section{Analytic solution in simple case}\label{sub:main_theorem}
With the solution \ref{eq:formula_3} of the Maxwell's equation with initial functions in the subspace $V^{\bw}$, we can derive an close form solution in a general setting as follows.
\begin{theo} Assume that $F_0(\bx)$ is periodic and can be expressed by its Fourier Series 
	%\begin{equation} \label{eq:main_init}
		\[
			F_0(\bx) = f_0 + \sum_{\bw\neq 0}  a_{\bw} e^{i(\bx\cdot \bw)} = c_0 + \sum_{\bw \neq 0}  \sum_{1\le d\le 3}  \alpha_{d;\bw} {\bv}_{d,\bw} e^{i(\bx\cdot \bw)}  =:c_0 + \sum_{\bw \neq 0} F_{0,\bw}(\bx),
		\]
	%\end{equation}
where $f_0\in C^3$ is constant and $\alpha_{d;\bw}$ is determined by Equation (\ref{eq:alpha}). Further assume that 
	\begin{equation}\label{eq:abs_convergence_condition}
		\sum_{\bw \neq 0} |\bw| |a_{\bw}| < \infty 
	\end{equation}
	converges. 
	Then Equation (\ref{maxwell_complex}) is solved by
	\begin{equation}\label{eq:main_formula}
		\FF(\bx,t)= f_0 + \sum_{\bw \neq 0}\sum_{1\le d\le 3} e^{\frac{-it \mu_{d,\bw}}{\sqrt{\mu\epsilon}}} \alpha_{d;\bw}  {\bv}_{d,\bw}e^{i(\bx\cdot \bw)}.
	\end{equation}
\end{theo}
\begin{proof} One can follow the derivation $F_{\alpha_{\bw}, \bw}(t,\bx)$ in Equation (\ref{eq:formula_3}) to formally obtain  the solution by Equation (\ref{eq-acp}).
	\begin{eqnarray*}
		\FF(\bx,t)=	e^{-itA} F_0 &=& f_0 + \sum_{\bw \neq 0}\sum_{1\le d\le 3} e^{\frac{-i t \mu_{d,\bw}}{\sqrt{\mu\epsilon}}} \alpha_{d;\bw}  {\bv}_{d,w}e^{i(\bx\cdot \bw)} .
	\end{eqnarray*}
	Under the assumption (\ref{eq:abs_convergence_condition}), $\FF_0$ is continuously differentiable. One can check directly that Equation (\ref{eq:main_formula}) indeed solves \eqref{maxsys} by passing relevant first order derivative operators into the summation in Equation (\ref{eq:main_formula}). The uniqueness of the solution is well-known due to the smoothness of $\FF_0$ \cite{Leis86} .  
\end{proof}
Since $\lambda_{\bw;1}=0$, we can rewrite Equation (\ref{eq:main_formula}) as 
%\begin{equation}\label{eq:main_formula_v2}
\[
\FF(\bx,t)-F(\bx,0)=\sum_{\bw \neq 0}\sum_{2\le d\le 3} (e^{\frac{-it \mu_{d,\bw}}{\sqrt{\mu\epsilon}}}-1) \alpha_{d;\bw}  {\bv}_{d,\bw}e^{i(\bx\cdot \bw)}.
\]
%\end{equation}
Let ${\bv}_{d,\bw}$ be defined by Equation  (\ref{v_d_w_1}) or Equation (\ref{v_d_w_2}).  It is easy to check that 
\begin{equation}\label{eq:cor}
	\nabla\cdot  ({\bv}_{d,\bw}e^{ix\cdot \bw})=0, \quad  d\in \{2,3\},
\end{equation}
which implies 
\begin{cor}\label{cor:conservation-divergence}
	$H$ and $E$ keep the conservation of divergence, i.e. 
	%\begin{equation}\label{maxsys b}
	\[
	\nabla\cdot (\EE - \EE_0)= 0, \qquad  \nabla\cdot  ( \HH - \HH_0)= 0.
	\]
	%\end{equation}
\end{cor}

\section{The analytic solution with $\bJ = J^g(\bx)$}\label{sec:jg}
In this section, we consider a setting by adding a generator term $J^g$, which is independently of the local electromagnetic field. To simplify the notation, we shall work in Heaviside-Lorentz units where $\eps=\mu= 1$ in the rest of the paper without loss of generosity, which is equivalent to replace $\sqrt \mu \HH$  by $\HH$,  $\sqrt{\eps}  \EE$ by $\EE$, and scale variable $t$ by $\sqrt{\mu\eps}$.  We aim to derive the analytic solution for following Maxwell equations.
\begin{eqnarray}
	\frac{\partial \HH}{\partial t} &=&-\CU (\EE), \qquad  \frac{\partial \EE}{\partial t} = \CU  (\HH) - J^g(\bx), \label{maxsys a-2}\\
	\HH(\bx,0) &=& H_0(\bx) ,\qquad  \EE(\bx,0) =E_0(\bx), \label{IV-2}
\end{eqnarray}
where the generator $J^g$ is stable and thus is independent on time $t$.

With the tools developed in previous sections,  it turns out that the desired solutions can be directly constructed as follows. 
\begin{theo}\label{theorem-main-3}
Let $J^g(\bx)$ be continuous differentiable and can be represented by its Fourier series as follows
\[
J^g(\bx) = \sum_{\bw} a_{\bw}e^{i\bx\cdot \bw} = v_0 + \sum_{\bw\neq 0} \sum_{1\le d \le 3} \alpha_{d,\bw}\bv_{d,\bw} e^{i\bx\cdot \bw},
\]
where $v_0 \in C^3$ is constant. Define 
\[
\phi(\bx) = v_0 + \sum_{\bw \neq 0} \alpha_{1,\bw}\bv_{1,\bw} e^{i\bx\cdot \bw} , \quad \psi(\bx)= \sum_{\bw \neq 0} \sum_{2\le d \le 3} \frac{1}{\mu_{d,\bw}} \alpha_{d,\bw}\bv_{d,\bw} e^{i\bx\cdot \bw}.
\]
Let $U, V$ be the solution without the generator $J^g$ 
\begin{eqnarray*}
	\frac{\partial V}{\partial t} &=&-\CU (U), \qquad  \frac{\partial U}{\partial t} = \CU  (V), \label{maxsys a-home}\\
	V(\bx,0) &=& H_0(\bx) - \psi(\bx) ,\qquad  U(\bx,0) = E_0(\bx)
\end{eqnarray*}
Then the desired solution for (\ref{maxsys a-2}) and (\ref{IV-2}) is given by
%\begin{equation}
\[
\EE= U(\bx, t)- t\phi(\bx), \quad \HH = V(\bx, t) +\psi(\bx).	
\]
\end{theo}
\begin{proof} It is clear that $(\EE,\HH)$ meets the initial conditions.
It is easy to see that
\begin{eqnarray*}
\frac{\partial \phi} {\partial t} &=&0, \quad  \CU  (\phi) =0, \\
\frac{\partial \psi} {\partial t}&=&0, \quad
\CU (\psi) =\sum_{\bw \neq 0} \sum_{2\le d \le 3}  \alpha_{d,\bw}\bv_{d,\bw} e^{i\bx\cdot \bw}  = J^g(\bx) -  \phi(\bx).
\end{eqnarray*}
We have
\begin{eqnarray*}
\dot \EE &=& \dot U - \phi(\bx) = \CU (V)- \phi(\bx) = \CU (\HH-\psi) - \phi(\bx)  = \CU (\HH) - J^g\\
\dot \HH &=& \dot V = -\CU (U) = - \CU (\EE + t\phi) = -\CU (\EE),
\end{eqnarray*}	
which concludes that $(\EE,\HH)$ are the desired solution.
\end{proof}

\section{Two demonstrative examples}\label{sec:example}
In this section, we provide two examples to demonstrate how to apply Equation (\ref{eq:main_formula}) to construct the analytic solutions of Maxwell's equations \eqref{maxsys} associated to the two types of eigenvector base discussed in Lemma \ref{lem-eigenvalues}.  We shall assume that $\mu=\epsilon=1$ in the rest of this section. 
\subsection{Eigenvector base with ${\br}_{\bw}=0$}
We aim to recover the solution of Equation (\ref{eq:main_formula}) with the initial conditions 
\begin{eqnarray*}
E(\bx, 0) &=&  \cos(2\pi(x_1+x_2+x_3))(1, -2, 1),\\
H(\bx, 0) &=&  \cos(2\pi(x_1+x_2+x_3))(\sqrt 3, 0, -\sqrt 3).
\end{eqnarray*}
The problem was studied in \cite{n5} and the analytical solution is known as 
\begin{eqnarray}
E(\bx, t) &=&  \cos(2\pi(x_1+x_2+x_3) -2\sqrt{3}\pi t)(1, -2, 1)\label{eq:sample_1_E} \\
H(\bx, t) &=&  \cos(2\pi(x_1+x_2+x_3) -2\sqrt{3}\pi t)(\sqrt 3, 0, -\sqrt 3)\label{eq:sample_1_H}.
\end{eqnarray}
Let $F(\bx, t)= E(\bx, t)+H(\bx, t)i$.  As the first step,  we represent $F(\bx, 0)$ in its Fourier series.  Let ${\bw}=\omega(1, 1,1)$ with $\omega=2\pi$,
\begin{eqnarray*}
F(\bx, 0) &=& (1 + i\sqrt 3 , -2, 1 - i\sqrt 3)^T \cos(2\pi(x_1+x_2+x_3) \\	
&=& 
	\left( \begin{array}{c}
		 \frac{1 +i\sqrt 3 }2\\
		-1\\
		\frac{1-i\sqrt 3}2\\
	\end{array}\right)  e^{i\bx \cdot \bw} + 
		\left( \begin{array}{c}
		\frac{1 + i\sqrt 3}2\\
		-1\\
		\frac{1 - i\sqrt 3}2\\
	\end{array}\right) \bar e^{i\bx \cdot \bw} =: F_1 + F_2.
\end{eqnarray*}
and $F_1(\bx, t),F_2(\bx, t)$ be the solutions with initial value $F_1$ and $F_2$ respectively. 
Applying $\ba =(\frac{1+ i\sqrt 3}2, -1, \frac{1- i\sqrt 3}2) $ to Equation (\ref{eq:alpha}), 
we obtain 
	\begin{equation*}
		\alpha_{\bw}=	V^{\dagger}_{\bw}\left( \begin{array}{c}
			a_{1,\bw} \\
			a_{2,\bw} \\
			a_{3,\bw} \\
		\end{array}\right) = \frac1{\sqrt 3} \left(
		\begin{array}{ccc}
		1 & 1 & 1\\
		\frac{1- i\sqrt 3}{2} & \frac{1+ i\sqrt 3}{2} & -1 \\
		\frac{1+ i\sqrt 3}{2} & \frac{1- i\sqrt 3}{2} &  -1
		\end{array}
		\right) \left( \begin{array}{c}
			\frac{1 + i\sqrt 3}2\\
			-1\\
			\frac{1 - i\sqrt 3}2\\
		\end{array}\right) = \sqrt 3 \left( \begin{array}{c}
		0\\
		0\\
		\frac{-1 + i\sqrt 3}2\\
		\end{array}\right).
	\end{equation*}
Hence 
\[
F_1(\bx, t) = \alpha_{3,\bw} v_{3,\bw} e^{i \bx \bw - i\mu_{3,\omega} t} = \frac{-1 + i\sqrt 3}2  \left( \begin{array}{c}
	\frac{1 - i\sqrt 3}2\\
	\frac{1 + i\sqrt 3}2\\
	-1\\
\end{array}\right) e^{i \bx \bw - i\mu_{3,\omega} t}  = \left( \begin{array}{c}
\frac{1 + i\sqrt 3}2\\
-1\\
\frac{1 - i\sqrt 3}2\\
\end{array}\right)e^{i \bx \bw - i\mu_{3,\omega} t},
\]
with 
\[
\mu_{3,\omega} = |\bw| = 2\pi \sqrt 3.  
\]
To cope with $F_2(\bx,t)$,  let $\hat\bw = \hat\omega(1,1,1)$ with $\hat\omega=-2\pi$. As in above treatment on $F_1(\bx,t)$, we have
\begin{equation*}
	\alpha_{\hat\bw}=	V^{\dagger}_{\hat\bw}\left( \begin{array}{c}
		a_{1,\hat\bw} \\
		a_{2,\hat\bw} \\
		a_{3,\hat\bw} \\
	\end{array}\right) = \frac1{\sqrt 3} \left(
	\begin{array}{ccc}
		1 & 1 & 1\\
		\frac{1+ i\sqrt 3}{2} & \frac{1- i\sqrt 3}{2} &  -1 \\
		\frac{1- i\sqrt 3}{2} & \frac{1+ i\sqrt 3}{2} & -1
	\end{array}
	\right) \left( \begin{array}{c}
		\frac{1 + i\sqrt 3}2\\
		-1\\
		\frac{1 - i\sqrt 3}2\\
	\end{array}\right) = \sqrt 3 \left( \begin{array}{c}
		0\\
		\frac{-1 + i\sqrt 3}2\\
		0
	\end{array}\right).
\end{equation*}
So
\[
F_2(\bx, t) = \alpha_{2,\hat\bw} v_{2,\hat\bw} e^{i \bx \hat\bw - i\mu_{2,\hat\omega} t} =\alpha_{3,\bw} v_{3,\bw} e^{i \bx \hat\bw - i\mu_{2,\hat\omega} t} = \left( \begin{array}{c}
	\frac{1 + i\sqrt 3}2\\
	-1\\
	\frac{1 - i\sqrt 3}2\\
\end{array}\right) e^{i \bx \hat\bw - i\mu_{2,\hat\omega} t} 
\]
with 
\[
\mu_{2,\hat \omega} = -|\hat\bw| = -2\pi \sqrt 3.
\]
Hence
\[
F_1(\bx, t) + F_2(\bx, t) = \left( \begin{array}{c}
	{1 + i\sqrt 3}\\
	-2\\
	{1 - i\sqrt 3}\\
\end{array}\right) \cos(2\pi(x_1+x_2+x_3-\sqrt 3 t)),
\]
which recoveries the solution  (\ref{eq:sample_1_E}) and (\ref{eq:sample_1_H}).

\subsection{Eigenvector base with ${\br}_{\bw}\neq 0$}
The second demonstrative example has the following initial condition
\[
F_0(\bx) = \cos({\bw}\cdot (\bx))(1,1,1)^T + \sin({\bw}\cdot (\bx))(1,1,1)^T i = (1,1,1) e^{i\bx \cdot \bw}
\]
with ${\bw}=(\pi, 2\pi, -3\pi)$. Applying $\ba =(1,1,1)^T$, we obtain
\[
	\mu_{1,\bw}=0, \quad \mu_{2,\bw}= -\pi \sqrt{14}=-\mu_{3,\bw}, \quad \nu_{\bw} = \sqrt 6, \quad \gamma_{\bw}= \pi \sqrt{42}, \quad {\br}_{\bw} = \pi (5, -4,-1)^T
\]
and
\begin{eqnarray*} 	
	v_{1,\bw}= \left( \begin{array}{c}
		\frac{1}{\sqrt{14}}\\
		\frac{2}{\sqrt{14}}\\
		\frac{-3}{\sqrt{14}}\\
	\end{array}\right), \quad
	v_{2,\bw}= \frac{1}{\nu_{\bw}}\left( \begin{array}{c}
		-1 +5\frac{i}{\sqrt{14}} \\
		-1 -4\frac{i}{\sqrt{14}}\\
		-1 -\frac{i}{\sqrt{14}}\\
	\end{array}\right), \quad
	v_{3,\bw}= \bar v_{2,\bw}
\end{eqnarray*}
\begin{eqnarray*} 
	\left( \begin{array}{c}
		\alpha_1\\
		\alpha_2\\
		\alpha_3\\
	\end{array}\right) = \frac1{\sqrt {14}}
	\left(
	\begin{array}{ccc}
		1 & 2 & -3\\
		-\frac{\sqrt{14}}{\sqrt 6}-\frac{5}{\sqrt 6}i  & -\frac{\sqrt{14}}{\sqrt 6}+\frac{4}{\sqrt 6}i &  -\frac{\sqrt{14}}{\sqrt 6}+\frac{1}{\sqrt 6}i  \\
		-\frac{\sqrt{14}}{\sqrt 6}+\frac{5}{\sqrt 6}i  & -\frac{\sqrt{14}}{\sqrt 6}-\frac{4}{\sqrt 6}i &  -\frac{\sqrt{14}}{\sqrt 6}-\frac{1}{\sqrt 6}i
	\end{array}
	\right)	\left(
	\begin{array}{ccc}
	1 \\
	1\\
	1	
	\end{array}
	\right)=\left(
	\begin{array}{ccc}
		0 \\
		-\sqrt{3/2}\\
		-\sqrt{3/2}
	\end{array}
	\right).
\end{eqnarray*}
Hence the solution is 
\[ 
	F(\bx,t)= -\sqrt {\frac{3}{2}} v_{2,\bw} e^{i({\bw}\cdot \bx - \mu_{2,\bw} t )} -\sqrt {\frac{3}{2}} v_{3,\bw}  e^{i({\bw}\cdot \bx - \mu_{3,\bw}t )}.
\]
Plugging the values of $ \mu_{2,\bw},  \mu_{3,\bw}$, we have
\begin{eqnarray*}
E(\bx,t) &=& \cos({\bw}\cdot \bx) \cos{\sqrt{14}\pi t}\left( \begin{array}{c}
	1\\
	1\\
	1\\
\end{array}\right) + \frac{1}{\sqrt{14}} \cos({{\bw}\cdot \bx}) \sin(\sqrt{14}\pi t) \left( \begin{array}{c}
5\\
-4\\
-1\\
\end{array}\right),\\
H(\bx,t) &=& \sin({\bw}\cdot \bx) \cos{\sqrt{14}\pi t}\left( \begin{array}{c}
	1\\
	1\\
	1\\
\end{array}\right) + \frac{1}{\sqrt{14}} \sin({{\bw}\cdot \bx}) \sin(\sqrt{14}\pi t) \left( \begin{array}{c}
	5\\
	-4\\
	-1\\
\end{array}\right).
\end{eqnarray*}
One can directly check that above $H$ and $E$ solve Equation (\ref{maxsys}) by applying 
\[
\CU \cos(\bw \cdot \bx) \left( \begin{array}{c}
	5\\
	-4\\
	-1\\
\end{array}\right) = 14\pi \sin(\bw \cdot \bx) \left( \begin{array}{c}
1\\
1\\
1\\
\end{array}\right), \quad \CU \sin(\bw \cdot \bx) \left( \begin{array}{c}
5\\
-4\\
-1\\
\end{array}\right) = -14\pi \cos(\bw \cdot \bx) \left( \begin{array}{c}
1\\
1\\
1\\
\end{array}\right),
\]
\[
\CU \cos(\bw \cdot \bx) \left( \begin{array}{c}
	1\\
	1\\
	1\\
\end{array}\right) = -\pi \sin(\bw \cdot \bx) \left( \begin{array}{c}
5\\
-4\\
-1\\
\end{array}\right), \quad \CU \sin(\bw \cdot \bx) \left( \begin{array}{c}
1\\
1\\
1\\
\end{array}\right) = \pi \cos(\bw \cdot \bx) \left( \begin{array}{c}
5\\
-4\\
-1\\
\end{array}\right).
\]
\appendix
\section{The proves of Proposition \ref{prop:property_maxwell_1}}
\subsection{The proof of Proposition \ref{prop:property_maxwell_1}}\label{app:property_maxwell_1}
\begin{proof}
 We follow the notations in Section \ref{sec:bj=0}. By definition, 
	\begin{eqnarray*}
		& & 4 |E_{\alpha_{\bw}, \bw}(\bx,t)|^2 = (F_{\alpha_{\bw}, \bw}(\bx,t) + \bar F_{\alpha_{\bw}, \bw}(\bx, t)) \cdot (F_{\alpha_{\bw}, \bw}(\bx,t) + \bar F_{\alpha_{\bw}, \bw}(\bx, t)) \\
		&=& (\alpha_{1,\bw}e^{i\bx\cdot \bw} + \bar\alpha_{1,\bw}  e^{-i\bx\cdot \bw})^2 
		+ 2(\alpha_{2,\bw}e^{i\bx\cdot \bw}  + \bar\alpha_{3,\bw} e^{-i\bx\cdot \bw}) (\alpha_{3,\bw}e^{i\bx\cdot \bw}  + \bar\alpha_{2,\bw} e^{-i\bx\cdot \bw}) \\ 
		&=& 2\sum_{1\le d\le 3}|\alpha_{d,\bw}|^2  + I + \bar I
	\end{eqnarray*}
	where
	\[
	I = (\alpha_{1,\bw}^2 + 2  \alpha_{2,\bw}\alpha_{3,\bw}) e^{2i\bx \cdot \bw}.
	\]
	Similarly, 
	\begin{eqnarray*}
		& & -4 |H_{\alpha_{\bw}, \bw}(\bx,t)|^2 = (F_{\alpha_{\bw}, \bw}(\bx,t) - \bar F_{\alpha_{\bw}, \bw}(\bx, t)) \cdot (F_{\alpha_{\bw}, \bw}(\bx,t)- \bar F_{\alpha_{\bw}, \bw}(\bx, t)) \\
		&=& (\alpha_{1,\bw}e^{i\bx\cdot \bw} - \bar\alpha_{1,\bw}  e^{-i\bx\cdot \bw})^2 
		+ 2(\alpha_{2,\bw}e^{i\bx\cdot \bw}  - \bar\alpha_{3,\bw} e^{-i\bx\cdot \bw}) (\alpha_{3,\bw}e^{i\bx\cdot \bw}  - \bar\alpha_{2,\bw} e^{-i\bx\cdot \bw}) \\ 
		&=& -2\sum_{1\le d\le 3}|\alpha_{d,\bw}|^2  + I + \bar I.
	\end{eqnarray*}
	and
	\begin{eqnarray*}
		&& 4i E_{\alpha_{\bw}, \bw}(\bx,t) \cdot H_{\alpha_{\bw}, \bw}(\bx,t) 
		= (F_{\alpha_{\bw}, \bw}(\bx,t) + \bar F_{\alpha_{\bw}, \bw}(\bx, t)) \cdot (F_{\alpha_{\bw}, \bw}(\bx,t)- \bar F_{\alpha_{\bw}, \bw}(\bx, t)) \\
		&=& \alpha^2_{1,\bw}e^{2i\bx\cdot \bw} -\bar\alpha^2_{1,\bw}  e^{-2i\bx\cdot \bw}
		+ (\alpha_{2,\bw}e^{i\bx\cdot \bw} + \bar\alpha_{3,\bw} e^{-i\bx\cdot \bw})(\alpha_{3,\bw}e^{i\bx\cdot \bw}  - \bar\alpha_{2,\bw} e^{-i\bx\cdot \bw}) \\ 
		&+& (\alpha_{2,\bw}e^{i\bx\cdot \bw}  - \bar\alpha_{3,\bw} e^{-i\bx\cdot \bw}) (\alpha_{3,\bw}e^{i\bx\cdot \bw}  + \bar\alpha_{2,\bw} e^{-i\bx\cdot \bw}) \\ 
		&=&  I -\bar I.
	\end{eqnarray*}
	So we obtain 
	\[
	|E_{\alpha_{\bw}, \bw}(\bx,t)|^2 + |H_{\alpha_{\bw}, \bw}(\bx,t)|^2 = \sum_{1\le d\le 3}|\alpha_{d,\bw}|^2
	\]
	is constant; $ |E_{\alpha_{\bw}, \bw}(\bx,t)|$, $|H_{\alpha_{\bw}, \bw}(\bx,t)|$ and $E_{\alpha_{\bw}, \bw}(\bx,t) \cdot H_{\alpha_{\bw}, \bw}(\bx,t)$ are stationary. 
\end{proof}

\end{document}